\documentclass[12pt,epsfig,amsfonts]{amsart}
\usepackage{amsmath,amsthm,amssymb,amscd,epsfig,color}
\usepackage{graphicx}
\usepackage{mathrsfs}
\usepackage{pb-diagram} 
\newcommand{\1}{\mbox{1}\hspace{-0.25em}\mbox{l}}

\newtheorem*{theorema}{Theorem A}
\newtheorem*{theoremb}{Theorem B}
\newtheorem*{theoremc}{Theorem C}

\newtheorem{prop}{Proposition}[section]
\newtheorem{lemma}[prop]{Lemma}

\newtheorem{cor}[prop]{Corollary}

\theoremstyle{definition}

\theoremstyle{remark}

\numberwithin{equation}{section}

\begin{document}

\author{Hiroki Takahasi and Kenichiro Yamamoto}

\address{Keio Institute of Pure and Applied Sciences (KiPAS), Department of Mathematics,
Keio University, Yokohama,
223-8522, JAPAN} 
\email{hiroki@math.keio.ac.jp}
\address{Department of general education, Nagaoka University of Technology, Nagaoka 940-2188, JAPAN}
\email{k\_yamamoto@vos.nagaokaut.ac.jp}



\title[Heterochaos baker maps and the Dyck system]
 {Heterochaos baker maps and the Dyck system:\\   maximal entropy measures and a mechanism for 
 the breakdown of entropy approachability} 
 \maketitle

 \begin{abstract}
 We introduce two parametrized families of piecewise affine maps on $[0,1]^2$ and $[0,1]^3$, as generalizations of the {\it heterochaos baker maps} which were introduced and investigated in  [Y. Saiki, H. Takahasi, J. A. Yorke, Nonlinearity, {\bf 34} (2021), 5744--5761] as minimal models of the unstable dimension variability in multidimensional dynamical systems.
   We show that  natural coding spaces of these maps 
coincide with the Dyck system that has come from the theory of languages.
Based on this coincidence, we start to develop 
a complementary analysis on their invariant measures.
As a first attempt, 
we show the existence of two ergodic measures of maximal entropy for the generalized
heterochaos baker maps. We also clarify a mechanism for the breakdown of entropy approachability.

    \end{abstract}

    \section{Introduction}
    Let $T$ be a Borel map acting on a compact metric space and
    let $\mathcal M(T)$ denote the set of $T$-invariant Borel probability measures. 
    For each $\mu\in\mathcal M(T)$, let $h(\mu)$ denote the measure-theoretic entropy of $\mu$ relative to $T$.
     If 
    $\rho(T)=\sup \{h(\mu)\colon\mu\in\mathcal M(T)\}$ is finite,
    a measure which attains this supremum is called a 
    measure of maximal entropy (mme).
         Dynamical systems with the unique measure of maximal entropy are called
     intrinsically ergodic. Intrinsically ergodic systems are expected to have nice properties, such as equidistribution of periodic points  \cite{Bow71}, exponential decay of correlations  \cite{Bow75,Rue04,Sin72} and so on.
   An important problem is to establish the intrinsic ergodicity for as large a set of dynamical systems as possible. 
          Another important problem is to investigate systems which are not intrinsically ergodic: to ask how many ergodic mmes coexist and what are their properties, and so on.
     These questions are related to phase transitions in statistical mechanics \cite{Rue04}.

              Most of the non-trivial 
              concrete examples of non-intrinsically ergodic systems 
              are subshifts on
              finite symbols.
  Krieger \cite{Kri74} 
     showed that 
   the Dyck system on $2m$ symbols, $m\geq2$ has exactly two ergodic measures of maximal entropy $\log (m+1)$,
   both of which are Bernoulli and fully supported.
      For other subshift examples including classical ones, see e.g.
     \cite{CliPav19,DGS76,GK18,Hay13,KOR16,Pav16,Pet86,Tho06}.

   The two-sided full shift on two symbols is 
   topologically (semi-)conjugate to the baker map or Smale's horseshoe map.
Here we address the following question. Let $\Sigma$ be a transitive subshift with positive topological entropy that is non-intrinsically ergodic.
Does there exist a piecewise differentiable map that is topologically semi-conjugate to $\Sigma$? 
     If so,
   one can bring techniques and results in differentiable dynamics
   into the analysis of 
   the coexisting mmes, which will certainly lead to a deeper 
   understanding of 
   the phenomenon of non-intrinsic ergodicity.
  The fact is that to construct semi-conjugacies to one-dimensional maps is difficult,
  due to the result of Hofbauer \cite{Hof79,Hof81} which asserts that a large class of piecewise monotonic interval maps are intrinsically ergodic.
  To construct
  semi-conjugacies to
   piecewise affine surface homeomorphisms or surface diffeomorphisms 
    seems also difficult, in view of the results of Buzzi et al. \cite{Bu09,BCS22}.
 So, let us modify our key question as follows. Let $\Sigma$ be a transitive subshift with positive topological entropy that is non-intrinsically ergodic.
   Does there exist a piecewise differentiable map whose ``symbolic model'' is $\Sigma$? 



    In this paper we show that 
    the Dyck system
    is 
    a natural coding space of
    simple piecewise affine maps on $[0,1]^2$ or $[0,1]^3$.
   These maps are 
 generalizations of the {\it heterochaos baker maps} introduced by Saiki et al. \cite{STY21} for a different purpose.
 It is interesting that the two systems with completely different origins are related in this way.
On the basis of this relation we start to develop a complementary analysis on invariant probability measures of the two systems.
  Below
  we introduce generalized heterochaos baker maps, the Dyck system, and state main results.

 \subsection{Generalized heterochaos baker maps}
  Let $m\geq2$ be an integer and put
 $c_0=\frac{1}{m}$.
For $a\in(0,c_0)$
 let $F_a\colon[0,1]\circlearrowleft$ be given by 
\[F_a(x)=\begin{cases}
\displaystyle{\frac{x-(i-1)a}{a}}&\text{ on }[(i-1)a,ia),\ i\in\{1,\ldots,m\},\\ \displaystyle{\frac{x-ma}{1-ma}}&\text{ on }
[ma,1].
\end{cases}\]
Let $\Omega_i^+$, $i\in\{1,\ldots,2m\}$ be pairwise disjoint domains in $[0,1]^2$ given by
\[\Omega_i^+=\begin{cases}\left[(i-1)a,ia\right)\times
\left[0,1\right]&\text{ for }i\in\{1,\ldots,m\},\\
\left[ma,1\right]\times
\left[\frac{i-m-1}{m },\frac{i-m }{m }\right)&
\text{ for }i\in\{m+1,\ldots,2m-1\},\\
\left[ma,1\right]\times
\left[\frac{i-m-1 }{m},1\right]&\text{ for }i=2m.
\end{cases}\]
Note that $[0,1]^2=\bigcup_{i=1}^{2m}\Omega_i^+$.
Define a map $f_{a}\colon [0,1]^2\circlearrowleft$  by
\[\begin{split}
  f_a(x,y)=
  \begin{cases}
    \displaystyle{\left(F_a(x),\frac{y}{m}+\frac{i-1}{m}\right)}&\text{ on }\Omega_i^+,\ i\in\{1,\ldots,m\},\\
   \displaystyle{\left (F_a(x),my-i+m+1\right)}&\text{ on }\Omega_i^+,\ i\in\{m+1,\ldots,2m\}.
   \end{cases}
\end{split}\]
See FIGURE~1.
As $a\to0$ or $a\to c_0$, the map degenerates to the baker map.

Next,
put
$\Omega_i=\Omega_i^+\times
\left[0,1\right]$ and 
let $a,b\in(0,c_0)$.
Define a map $f_{a,b}\colon [0,1]^3\circlearrowleft$ by 
\[\begin{split}
  f_{a,b}(x,y,z)=
  \begin{cases}
    \displaystyle{\left(f_{a,b}(x,y),
    (1-mb)z\right)}&\text{ on }\Omega_i,\ i\in\{1,\ldots,m\}
    ,\\
   \displaystyle{\left (f_{a,b}(x,y),bz+1-mb+b(i-m-1)\right)}&\text{ on }\Omega_i,\ i\in\{m+1,\ldots,2m\}.
   \end{cases}
\end{split}\]
 See FIGURE~2. 
 The map $f_a$ is the projection of $ f_{a,b}$ to the $xy$-plane, and $ f_{a,b}$ is invertible while $f_a$ is not. 
By the symmetry $f_{a,b}^{-1}(x,y,z)= f_{b,a}(-z+1,y,x)$,
it does not lose generality to restrict to $(a,b)$ with $a\geq b$.
We call $f_a$ and $f_{a,b}$
{\it generalized heterochaos baker maps}.
If the context is clear,
we will drop $a,b$ from notation.

\begin{figure}
\begin{center}
\includegraphics[height=3.4cm,width=9.5cm]
{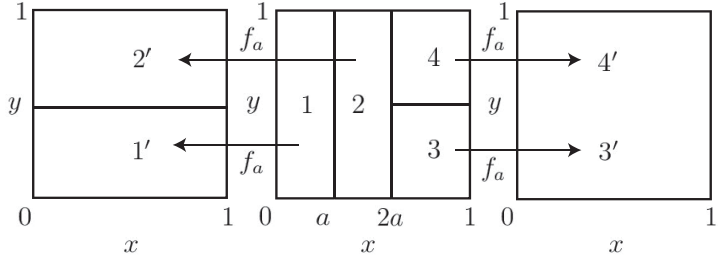}
\caption
{The map $f_a$
for $m=2$. The domains $\Omega_1^+$, $\Omega_2^+$, $\Omega_3^+$, $\Omega_4^+$ are labeled with $1$, $2$, $3$, $4$
and their images 
are labeled with $1'$, $2'$, $3'$, $4'$
respectively: $f_a(\Omega_3^+)=[0,1]\times[0,1)$ and $f_a(\Omega_4^+)=[0,1]^2$.}
\end{center}
\end{figure}\label{fig}

 \begin{figure}
\begin{center}
\includegraphics[height=4cm,width=9.5cm]
{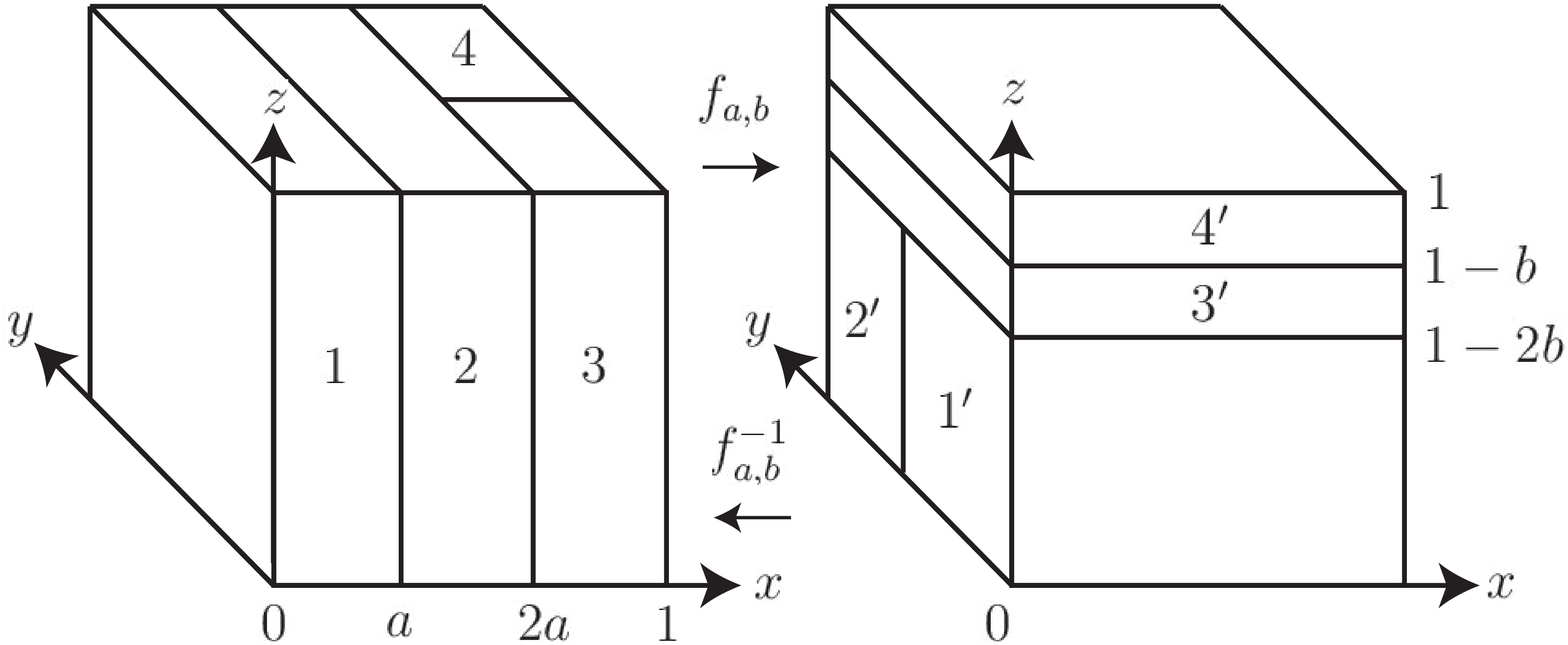}
\caption
{The map $f_{a,b}$ and its inverse $f_{a,b}^{-1}$ for $m=2$. The domains $\Omega_1$,
$\Omega_2$, $\Omega_3$, $\Omega_4$ are labeled with $1$, $2$, $3$, $4$
and their images under $f_{a,b}$ are 
labeled with $1'$, $2'$, $3'$, $4'$ respectively.}
\end{center}\label{fig2}
\end{figure}

  
 The maps $f_{\frac{1}{3}}$ and $f_{\frac{1}{3},\frac{1}{6}}$ for $m=2$ 
were introduced in \cite{STY21} as minimal models of the {\it unstable dimension variability}
\cite{DGSY94,KKGOY}.
 For multidimensional dynamical systems, it often happens that
in some regions the dynamics is unstable in 
 more directions than in other regions. This suggests that the dimension
 of unstable directions, when defined, 
 is not constant and varies from point to point.
 The unstable dimension variability means this situation, which was first considered by Abraham and Smale \cite{AS70} and is now believed to be a
 ubiquitous phenomenon for multidimensional systems which are not uniformly hyperbolic.
 
     Under the iteration of $f=f_{a,b}$,
the $x$-direction is expanding by factor $\frac{1}{a}$ or $\frac{1}{1-ma}$,
 the $z$-direction is contracting by factor
$(1-mb)$ or $b$,
while the $y$-direction is central:
contracting by factor $1/m$ on $\bigcup_{i=1}^{m} \Omega_i$ and expanding by factor $m$ on $\bigcup_{i=m+1}^{2m} \Omega_{i}$.
In particular, $f$ is not a uniformly hyperbolic system.
The stability in the central direction along each orbit is determined by the asymptotic time average of 
the function 
which takes $-\log m$ on 
$\bigcup_{i=1}^{m} \Omega_i$ and
$\log m$ on
$\bigcup_{i=m+1}^{2m} \Omega_i$.
 If $p\in X$ and $n\geq1$,
 the average value of this function 
 over $\{f^{k}(p)\}_{k=0}^{n-1}$
  equals the
exponential growth rate of the second diagonal element of the Jacobian matrix of $f^n$ at $p$.
If $f^n(p)=p$, the unstable dimension of the periodic point $p$ (the dimension of the unstable subbundle of $T_p\mathbb R^3$) is either $1$ or $2$ according as the average value
is negative or positive. 
  For
   $f_{\frac{1}{3},\frac{1}{6}}$, there is a dense orbit 
   and both types of periodic points are dense 
\cite[Theorem~1.1]{STY21}.
Using the arguments in \cite{STY21}
one can show that these properties remain true for all $m\geq2$ and all $a,b\in(0,c_0)$.


\subsection{The Dyck system}\label{Dyck}
The theory of symbolic dynamics has a close relationship to automata theory and language theory. 
In the theory of languages, there is a universal language due to W. Dyck. 
The Dyck system \cite{AU68,Kri74} is
the symbolic dynamics generated by that language.
It is a subshift on the alphabet 
$D=\{\alpha_1,\ldots,\alpha_m,\beta_1,\ldots,\beta_m\}$
consisting of $m$ brackets, $\alpha_i$ left and $\beta_i$ right in pair,
whose admissible words are words of legally aligned brackets.
To be more precise,
let $D^*$ denote the set of finite words of elements of $D$.
Consider the monoid with zero, with $2m$ generators in $D$ and the unit element
$1$ with relations 
\[\alpha_i\cdot\beta_j=\delta_{i,j},\
0\cdot 0=0\text{ and }\] \[\gamma\cdot 1= 1\cdot\gamma=\gamma,\
 \gamma\cdot 0=0\cdot\gamma=0
\text{ for }\gamma\in D^*\cup\{ 1\},\]
where $\delta_{i,j}$ denotes Kronecker's delta.
For $n\geq1$ and $\gamma_1\gamma_2\cdots\gamma_n\in D^*$ 
let
\[{\rm red}(\gamma_1\cdots\gamma_n)=\prod_{i=1}^n\gamma_i.\]
The one and two-sided Dyck shifts on $2m$ symbols are respectively defined by
\[\begin{split}\Sigma_{D}^+&=\{\omega\in D^{\mathbb Z^+}\colon {\rm red}(\omega_i\cdots \omega_j)\neq0\text{ for }i,j\in\mathbb Z^+\text{ with }i<j\},\\
\Sigma_{D}&=\{\omega\in D^{\mathbb Z}\colon {\rm red}(\omega_i\cdots \omega_j)\neq0\text{ for }i,j\in\mathbb Z\text{ with }i<j\},\end{split}\]
where $\mathbb Z^+$ denotes the set of non-negative integers.
Note that $\Sigma_D$ is the natural extension of $\Sigma_D^+$.

\subsection{Statements of results}
As inferred from their definitions,
the generalized heterochaos baker maps 
have natural coarse symbolic representations.
We put
\[X_a=\bigcap_{n\in\mathbb Z^+}f_a^{-n}\left(\bigcup_{i=1}^{2m}{\rm int}( \Omega_i^+)\right)\quad\text{ and }\quad X_{a,b}=\bigcap_{n\in\mathbb Z}f_{a,b}^{-n}\left(\bigcup_{i=1}^{2m}{\rm int}(\Omega_i)\right),\]
 and
define {\it coding maps} 
 $\pi_{a}\colon (x,y)\in X_a\mapsto
(\omega_n)_{n\in\mathbb Z^+}\in\{1,\ldots,2m\}^{\mathbb Z^+}$ 
by \[(x,y)\in 
\bigcap_{n\in\mathbb Z^+}f_a^{-n}({\rm int}(
\Omega_{\omega_n}^+)),\]
 and $\pi_{a,b}\colon (x,y,z)\in X_{a,b}\mapsto
 (\omega_n)_{n\in\mathbb Z}\in\{1,\ldots,2m\}^{\mathbb Z}$ by
  \[(x,y,z)\in\bigcap_{n\in\mathbb Z} f_{a,b}^{-n}({\rm int}
  (\Omega_{\omega_n})).\]
   The corresponding coding spaces \[\Sigma_{H\!C}^+= {\rm cl}( \pi_{a}(X_a))\quad\text{and}\quad\Sigma_{H\!C}=
   {\rm cl}(\pi_{a,b}(X_{a,b}))\]
  are subshifts on $2m$ symbols which are independent of $(a,b)$. We call them {\it heterochaos shifts}.
  Let $\sigma^+\colon\Sigma_{H\!C}^+\circlearrowleft$ and $\sigma\colon\Sigma_{H\!C}\circlearrowleft$ denote the left shifts.
   We have $\sigma^+\circ\pi_{a}=\pi_{a}\circ f_{a}$ and  $\sigma\circ\pi_{a,b}=\pi_{a,b}\circ f_{a,b}.$
   Since $f_a$ and $f_{a,b}$ are transitive, $\Sigma_{H\!C}^+$ and $\Sigma_{H\!C}$
    are irreducible.
    The coding maps are continuous, 
  neither surjective (Lemma~\ref{cset})
 nor injective.
For example, there exist infinitely many continua of periodic points which collapse into a single periodic point in the coding spaces.

 We say a subshift
 $\Sigma_1\subset S_1^{\mathbb Z^+}$ is {\it isomorphic} to 
 a subshift $\Sigma_2\subset S_2^{\mathbb Z^+}$ if there exists a bijection
 $\kappa\colon S_1\to S_2$ such that the map  $(\omega_n)_{n\in\mathbb Z^+}\in \Sigma_1\mapsto (\kappa(\omega_n))_{n\in\mathbb Z^+}\in S_2^{\mathbb Z^+}$ is a homeomorphism from $\Sigma_1$ to $\Sigma_2$.
 If $\Sigma_1$ is isomorphic to $\Sigma_2$, we write $\Sigma_1\simeq \Sigma_2$. Isomorphisms between 
 two-sided subshifts are defined analogously.
 Our first result asserts that the heterochaos shifts are isomorphic to the Dyck shifts. 

\begin{theorema}[Realization of the Dyck system]
We have
$\Sigma_{H\!C}^+\simeq\Sigma_{D}^+$ and
$\Sigma_{H\!C}\simeq\Sigma_{D}$.
The isomorphism
$\kappa\colon\{1,\ldots,2m\}\to D$ is given by
 $\kappa(i)=\alpha_i$, $\kappa(i+m)=\beta_i$ for $i\in\{1,\ldots, m\}$.
\end{theorema}

 Concerning
mmes of $f_{a,b}$ we have the following result. 
The corresponding statement for $f_a$ is also true,
 since
the canonical projection $\Sigma_{D}\to\Sigma_{D}^+$ induces a homeomorphism from $\mathcal M(\sigma)$ to $\mathcal M(\sigma^+)$ that preserves ergodicity and entropy of invariant measures.
  Put $c_1=\frac{1}{m(m+1)}$,
 $c_2=\frac{1}{m+1}$.

     \begin{theoremb}[Measures of maximal entropy]
For all $a,b\in(0,c_0)$ with $a\geq b$,
there exist two $f_{a,b}$-invariant
 ergodic Borel probability measures $\mu_\alpha$, $\mu_\beta$ 
 of entropy $\log (m+1)$ which are Bernoulli, 
 give positive weight to any nonempty open subset of $[0,1]^3$, and satisfy
\[\mu_\alpha\left(\bigcup_{i=1}^m\Omega_i\right)=
\mu_\beta\left(\bigcup_{i=m+1}^{2m}\Omega_i\right)=\frac{m}{m+1}.\]
Moreover,
 if $a\in[c_1,c_2]$ 
then $\rho(f_{a,b})=\log (m+1)$.
If $a\in(c_1,c_2)$,
there is no  ergodic measure of maximal entropy for $f_{a,b}$ other than $\mu_\alpha$, $\mu_\beta$.
\end{theoremb}
In other words,
$\mu_\alpha$
is biased toward the symbols  $1,\ldots,m$,
and $\mu_\beta$ is biased toward the symbols $m+1,\ldots,2m$.
In fact,
for $a=c_1$ and $a=c_2$, one of the two mmes is the Lebesgue measure on $[0,1]^3$.
Statistical properties of the mmes for $f$ and the Dyck system will be treated in our forthcoming papers.

The half and half combination of the two mmes in Theorem~B
is totally balanced, and as a result has a peculiar approximation property by ergodic measures.
     For a sequence $\{\mu_n\}_{n=1}^\infty$ in $\mathcal M(T)$ which converges to $\mu\in\mathcal M(T)$ in the weak* topology, we 
write $\mu_n\to\mu$. 
 We say a non-ergodic measure $\mu\in\mathcal M(T)$ is
{\it entropy approachable by ergodic measures} if 
there exists 
a sequence $\{\mu_n\}_{n=1}^\infty$ of ergodic measures
in $\mathcal M(T)$
such that $\mu_n\to\mu$ and $\liminf_{n\to\infty}h(\mu_n)\geq h(\mu)$.
We say $\mu$ is
{\it approachable by ergodic measures} if it is contained in the closure of the set of ergodic measures.
\begin{theoremc}[Non-entropy approachability]
For all $a\in(c_1,c_2)$, $b\in(0,c_0)$ with $a\geq b$, the measure $\frac{1}{2}(\mu_\alpha+\mu_\beta)$
 is approachable and not entropy approachable by ergodic measures. 
\end{theoremc}

We say $T$ is {\it entropy dense}  if any non-ergodic measure in $\mathcal M(T)$ is entropy approachable by ergodic measures.
The entropy density is a useful property in theories such as
multifractal analysis and large deviations in which all invariant measures come into play, in order to reduce one's consideration to ergodic measures only.
Extensions of these theories to
non-entropy dense systems
will require more examples of non-ergodic measures 
which are approachable and not entropy approachable by ergodic measures.
Asaoka and the second-named author \cite{AY13} constructed such a measure for a $C^\infty$ non-transitive diffeomorphism on $\mathbb T^5$. 
Gelfert and Kwietniak \cite[Proposition~4.29]{GK18} constructed 
 a subshift example.
Theorem~C provides a new example of such measures.


\subsection{Structure of the paper}
The rest of this paper consists of two sections and one appendix.
In Section~2 we prove Theorem~A
by comparing labeled directed graphs associated with $\Sigma_D^+$ and $\Sigma_{H\!C}^+$.
In Section~3 we prove Theorems~B and C.
In the appendix we comment on the Motzkin system \cite{Ino10,M04}.

The proof of Theorem~B and that of Theorem~C are related and briefly outlined as follows.
In light of Theorem~A, let $\iota\colon \Sigma_{H\!C}\to\Sigma_D$ denote the isomorphism and put $\bar\pi=\iota\circ\pi$.
 We will denote the two ergodic mmes on $\Sigma_D$ by $\nu_\alpha$ and $\nu_\beta$,
and following \cite{Kri74} 
introduce pairwise disjoint shift invariant Borel subsets
$A_\alpha$, $A_\beta$ of $\Sigma_D$ such that 
$\nu_\alpha(A_\alpha)=1$ 
and $\nu_\beta(A_\beta)=1$. 
We show that $\bar\pi$ induces a bijection between 
$A_\alpha\cap\bar\pi(X)$ and $\bar\pi^{-1}(A_\alpha)$, and that between 
$A_\beta\cap\bar\pi(X)$ and $\bar\pi^{-1}(A_\beta)$ both of which
 preserve entropies and ergodicity of invariant measures
(Lemma~\ref{m-bij}).
We also show that 
$\nu_\alpha(\bar\pi(X))=1$ and
$\nu_\beta(\bar\pi(X))=1$.
Then $\nu_\alpha$ and $\nu_\beta$
are transferred to ergodic
invariant measures of $f$ of entropy $\log(m+1)$.

By \cite{Kri74},
any ergodic measure other than $\nu_\alpha$ and $\nu_\beta$ which give measure $1$ to the set $\bar\pi^{-1}(A_\alpha\cup A_\beta)$ has entropy strictly less than $\log(m+1)$.
Our task is to bound from above entropies of ergodic measures which give measure $0$ to $\bar\pi^{-1}(A_\alpha\cup A_\beta)$. To this end we introduce Lyapunov exponents (Section~\ref{lyapunov-sec}), and show that such measures have zero central Lyapunov exponent.
We prove a version of Ruelle's inequality \cite{Rue78} which bounds entropies of invariant measures in terms of their Lyapunov exponents (Lemma~\ref{ruelle-ineq}).
If the central Lyapunov exponent is sufficiently close to $0$ and $a\in(c_1,c_2)$, 
the unstable Lyapunov exponent does not exceed $\log(m+1)$ (Lemma~\ref{zero-u}).
Combining these two bounds yields
the desired entropy bound which verifies Theorem~B.
Using
this entropy bound, we also prove that 
{\it almost balanced non-ergodic measures with high entropy, like $\frac{1}{2}(\mu_\alpha+\mu_\beta)$, are not entropy approachable} (cf. Lemma~\ref{ent-ap}). 
To complete the proof of Theorem~C,
we exploit the symmetry in the Dyck shift and prove that 
$\frac{1}{2}(\mu_\alpha+\mu_\beta)$ is approachable by ergodic measures supported on a single periodic orbit.


\section{Symbolic dynamics}
 In this section we consider symbolic dynamics of the generalized heterochaos baker maps.
 Using labeled directed graphs introduced
 in Section~\ref{pre}, we prove Theorem~A in Section~\ref{pfthma}.
  \subsection{Graph representations of subshifts}\label{pre}
    A subshift $\Sigma$ on  a finite discrete set $S$ is a shift invariant closed subset of the Cartesian product topological space $S^{\mathbb Z^+}$ or 
    $S^{\mathbb Z}$.
 Let $L(\Sigma)$
 denote the set of finite words of elements of $S$ which appear in some elements of $\Sigma$.
 For convenience, we include the empty word $\emptyset$ in  $L(\Sigma)$
 and let $\emptyset \xi=\xi\emptyset=\xi$ for  $\xi\in L(\Sigma)\setminus\{\emptyset\}$.
  For
   $n\in\mathbb Z^+$, 
    let $L_n(\Sigma)$ denote the set of elements of
    $L(\Sigma)$
    with word length $n$.
   For $\gamma_0\cdots \gamma_{n-1}\in  L_n(\Sigma)$ define 
 \[ [\gamma_{0}\cdots \gamma_{n-1}]=\{\omega=(\omega_i)_i\in\Sigma\colon \omega_{i}=\gamma_i\text{ for  }0\leq i\leq n-1\}.\]

A {\it directed graph} 
consists of a countable set $V$ of {\it vertices} and a countable set $E$ of {\it edges},
together with two
maps $\iota\colon E\to V$,
 $\tau\colon E\to V$
that
assign to each edge $e\in E$
a pair 
$(\iota(e),\tau(e))\in V\times V$ of vertices.
We say $e$ starts at $\iota(e)$ and terminates at $\tau(e)$.
Let $(V,E)$ be a directed graph, $S$ be a finite set and let
$\varphi\colon E\to S$ be a map.
For
$v,v'\in V$ and $\gamma\in S$
we denote $v\xrightarrow{\gamma} v'$
if there exists $e\in E$
such that $\iota(e)=v$, $\tau(e)=v'$
and 
$\varphi(e)=\gamma$.
We call the triplet
$(V,E,\varphi)$
an {\it $S$-labeled directed graph}.
We call $\gamma_0\gamma_1\cdots\in S^{\mathbb Z^+}$ 
an {\it infinite labeled path} if there are infinitely many vertices $v_0,v_1,\ldots$ such that $v_n\xrightarrow{\gamma_n} v_{n+1}$ for $n\ge0$.

Let $\Sigma$ be a one-sided subshift on $S$.
A subset of $\Sigma$ of the form
\[\Theta^\xi=\{\eta\in\Sigma\colon \xi \eta\in\Sigma\},\ \xi\in L(\Sigma)\]
is called a {\it follower set}.
The {\it follower set graph} 
of $\Sigma$ is 
the $S$-labeled directed graph 
whose vertices are follower sets,
such that 
for two vertices
 $\Theta^\xi$, $\Theta^\eta$
and $\gamma\in S$,
$\Theta^\xi\xrightarrow{\gamma} \Theta^\eta$ 
if and only if $\xi\gamma\in L(\Sigma)$ and $\Theta^{\eta}=\Theta^{\xi\gamma}$.
In this graph, $v\xrightarrow{\gamma} v'$  implies that $v'$ equals the shift image of $v\cap[\gamma]$.
Therefore,
 $\Sigma$ coincides with the set of infinite labeled paths in the follower set graph of $\Sigma$.



\begin{prop}
\label{rule-graph}
The set of vertices in the follower set graph
 of $\Sigma_{H\!C}^+$ is
\[\{\Theta^\xi\colon\xi\in L(\{1,\ldots,m\}^
{\mathbb Z^+})\}.\]
The labeled edges in the graph are:
 \[\begin{split}&\Theta^\xi\xrightarrow{i}\Theta^{\xi i},\  \Theta^{\xi i}\xrightarrow{i+m} \Theta^\xi\quad \text{for }\xi\in L (\{1,\ldots,m\}^{\mathbb Z^+}),\ i\in\{1,\ldots,m\};\\ 
    &\Theta^{\emptyset}\xrightarrow{i}\Theta^{\emptyset}\quad\quad\quad\quad\quad\quad\quad \text{for }i\in\{m+1,\ldots, 2m\}.\end{split}\]
\end{prop}
\begin{proof}
By the definitions of $f_a$ and
$\Sigma_{H\!C}^+$,
for $n\geq1$,
 $\xi\in\{1,\ldots,m\}^n$ and $i\in\{1,\ldots,m\}$ we have
$\Theta^{\xi i(i+m)}=\sigma(\sigma(\sigma^n[\xi]\cap [i])\cap [i+m])$,
and $\Theta^i=\Theta^\emptyset$ for $i\in\{m+1,\ldots,2m\}$.
These equalities imply the assertions of the proposition.
\end{proof}

\subsection{Proof of Theorem~A}\label{pfthma}
By the relations in Section~\ref{Dyck}
we have $\Theta^{\xi\alpha_i\beta_i }=\Theta^\xi$
for $\xi\in L (\{\alpha_1,\ldots,\alpha_m\}^{\mathbb Z^+ })$ and $i\in\{1,\ldots,m\}$,
and
$\Theta^{\beta_i}=\Theta^\emptyset$
for $i\in\{1,\ldots,m\}$.
The set of vertices in the follower set graph
 of $\Sigma_{D}^+$ is
$\{\Theta^\xi\colon \xi\in L(\{\alpha_1,\ldots,\alpha_m\}^{
\mathbb Z^+ })\}.$
The labeled edges are:
$\Theta^\xi\xrightarrow{\alpha_i}\Theta^{\xi\alpha_i}$, $\Theta^{\xi\alpha_i}\xrightarrow{\beta_i} \Theta^{\xi}$  for $\xi\in L (\{\alpha_1,\ldots,\alpha_m\}^{\mathbb Z^+ }),\ i\in\{1,\ldots,m\}$; $\Theta^{\emptyset}\xrightarrow{\beta_i}\Theta^{\emptyset}$ for $i\in\{1,\ldots, m\}.$
From this and Proposition~\ref{rule-graph},
the follower set graph of $\Sigma_{H\!C}^+$ coincides with that of $\Sigma_{D}^+$ up to the replacement of symbols
$\kappa\colon\{1,\ldots,2m\}\to D$,
 $\kappa(i)=\alpha_i$, $\kappa(i+m)=\beta_i$ for $i\in\{1,\ldots, m\}$.
Hence
$\Sigma_{H\!C}^+\simeq\Sigma_{D}^+$ holds.
From this and Proposition~\ref{natural-lem2} below we obtain
$\Sigma_{H\!C}\simeq\Sigma_{D}$.\qed

\begin{prop}\label{natural-lem2}
$\Sigma_{H\!C}$ is the natural extension of $\Sigma_{H\!C}^+$.
\end{prop}
\begin{proof}
Fix $a,b\in(0,c_0)$.
It is easy to see that
 $\pi_{a,b}(X_{a,b})$ is the natural extension of $\pi_a(X_a)$ in the sense that \[\pi_{a,b}(X_{a,b} )=\{
  \omega\in \{1,\ldots,2m\}^{\mathbb Z}\colon (\omega_n)_{n\geq-k}\in \pi_a(X_a)\text{ for }k\geq1\}.\]
Taking closures of both sides of the equality we obtain
\[\Sigma_{H\!C}\subset\{\omega\in \{1,\ldots,2m\}^{\mathbb Z}\colon (\omega_n)_{n\geq-k}\in \Sigma_{H\!C}^+\text{ for }k\geq1\}.\]
The opposite inclusion can easily be checked.
\end{proof}

\section{Invariant measures}
In light of Theorem~A,  we also denote by 
$\sigma$ the left shift
on $\Sigma_D$.
In Section~\ref{Dyck-review} we review the structure of the Dyck shift and its two ergodic mmes from \cite{Kri74}.
In Section~\ref{range}  we identify the set of points which are out of the image of the coding map.
In Section~\ref{lyapunov-sec} we introduce Lyapunov exponents, and provide in Section~\ref{entropy} upper bounds on entropy of an ergodic measure of $f_{a,b}$ in terms of its Lyapunov exponents. In Section~\ref{per-sec} we show that the half and half combination of the ergodic mmes of the Dyck shift is approachable by ergodic measures supported on a periodic orbit. The proof of Theorem~B and that of Theorem~C are given in Sections~\ref{pfthmb} and \ref{pfthmc} respectively.

\subsection{Representation of the ergodic mmes of the Dyck shift}\label{Dyck-review}
Let $\alpha$, $\beta$ be new symbols
and define
 \[\Sigma_\alpha=\{\alpha_1,\ldots,\alpha_{m},\beta\}^{\mathbb Z}\quad\text{and}\quad\Sigma_\beta=\{\alpha,\beta_1,\ldots,\beta_m\}^{\mathbb Z}.\]
 Let $\lambda_\alpha$, $\lambda_\beta$ denote the Bernoulli measures on 
$\Sigma_{\alpha}$,
$\Sigma_{\beta}$
respectively. 
As shown in \cite[Section~4]{Kri74},
the two ergodic mmes for $\Sigma_D$ are isomorphic to 
$\lambda_\alpha$ or $\lambda_\beta$. For our purpose we make explicit the isormorphisms, with a slightly different notation from the one in \cite[Section~4]{Kri74}.

 For $i\in\mathbb Z$ define $H_i\colon \Sigma_D\to\mathbb Z$ by
      \[H_i(\omega)=\begin{cases}\sum_{j=0}^{i-1} \sum_{l=1}^{m}(\delta_{{\alpha_l},\omega_j}-\delta_{\beta_l,\omega_j})&\text{ for }  i\geq1,\\\sum_{j=i}^{-1} \sum_{l=1}^{m}(\delta_{{\beta_l},\omega_j}-\delta_{\alpha_l,\omega_j})&\text{ for } i\leq -1,\\
      0&\text{ for }i=0.\end{cases}\]
      These functions can be used to express whether a bracket in a prescribed position is closed or not.
       We will often use the following property: let $\omega\in\Sigma_D$ and let $i$, $j\in\mathbb Z$ satisfy $i<j$. Then
       $H_{i}(\omega)=H_j(\omega)$ if and only if ${\rm red}(\omega_{i}\cdots \omega_{j-1})=1$.
For simplicity we write as follows: $\{H_i=H_j\}=\{\omega\in\Sigma_D\colon H_{i}(\omega)=H_j(\omega)\}$;
$\Sigma_D(i,\gamma)=\{\omega\in\Sigma_D\colon\omega_i=\gamma\}$ for $\gamma\in D$;
$\Sigma_\alpha(i,\gamma)=\{\omega\in\Sigma_\alpha\colon\omega_i=\gamma\}$
for $\gamma\in\{\alpha_1,\ldots,\alpha_m,\beta\}$;
$\Sigma_\beta(i,\gamma)=\{\omega\in\Sigma_D\colon\omega_i=\gamma\}$ for  $\gamma\in\{\alpha,\beta_1,\ldots,\beta_m\}$.

    We introduce three shift invariant Borel sets
    \[\begin{split}A_0&=\bigcap_{i=-\infty}^\infty\left(\left(\bigcup_{l=1}^\infty\{ H_{i+l}=H_i\}\right)\cap\left(\bigcup_{l=1}^\infty\{ H_{i-l}=H_i\}\right)\right),\\
A_\alpha&=\left\{\omega\in\Sigma_D\colon
\lim_{i\to\infty}H_i(\omega)
=\infty\text{ and }\lim_{i\to-\infty}H_i(\omega)=-\infty\right\},\\
A_\beta&=\left\{\omega\in\Sigma_D\colon
\lim_{i\to\infty}H_i(\omega)=-\infty\text{ and }
\lim_{i\to-\infty}H_i(\omega)=\infty\right\}.\end{split}\]
Ergodic measures in $\mathcal M(\sigma)$ give measure $1$ to one of the three sets \cite[pp.102--103]{Kri74}.

Recall that $\bar\pi=\iota\circ\pi$ where $\iota\colon\Sigma_{H\!C}\to\Sigma_D$ is the isomorphism given by Theorem~A.
In order to know which measures in $\mathcal M(\sigma)$ can be transferred to measures in $\mathcal M(f)$ via the map $\bar\pi$, we introduce
a shift invariant Borel set
\[A=\left\{\omega\in\Sigma_D\colon \liminf_{i\to\infty}H_i(\omega)=-\infty\text{ or }\liminf_{i\to-\infty}H_i(\omega)=-\infty\right\}.\]
Note that $A_\alpha\cup A_\beta\subset A$.
\begin{lemma}\label{m-bij}
The restriction of 
$\bar\pi$ to $\bar\pi^{-1}(A)$ is a homeomorphism onto its image.
\end{lemma}
\begin{proof}
Let $\omega\in A$
satisfy
 $\bar\pi^{-1}(\omega)\neq\emptyset$.
 Suppose $\bar\pi^{-1}(\omega)$ is not a singleton.
 Then it contains a segment $l$ parallel to the $y$-axis such that
 $ f^i$ is affine on $l$
 for all $i\in\mathbb Z$.
 We have
${\rm diam}(f^i(l))= m^{-H_i(\omega)}{\rm diam}(l)$,
where ${\rm diam}$ denotes the Euclidean diameter.
If $\liminf_{i\to\infty}H_i(\omega)=-\infty$ 
then $\limsup_{i\to\infty}{\rm diam}(f^i(l))=\infty$.
If
$\liminf_{i\to-\infty}H_i(\omega)=-\infty$ then 
$\limsup_{i\to-\infty}{\rm diam}(f^i(l))=\infty$.
 In either of the cases we reach a contradiction.
Therefore,
 the restriction of $\bar\pi$ to $\bar\pi^{-1}(A)$ is injective.
 It is a homeomorphism since the inverse map is continuous.
\end{proof}

The two ergodic mmes of the Dyck shift give measure $1$ to sets smaller than $A_\alpha$ and $A_\beta$ respectively.
We introduce two shift invariant Borel sets
\[\begin{split}B_\alpha&=\bigcap_{i=-\infty}^\infty\bigcup_{l=1}^m\left(\Sigma_D(i,\alpha_l)\cup\bigcup_{k=1}^\infty\Sigma_D(i,\beta_l)\cap\{ H_{i-k+1}=H_{i+1}\}\right),\\
B_\beta&=\bigcap_{i=-\infty}^\infty\bigcup_{l=1}^m\left(\Sigma_D(i,\beta_l)\cup\bigcup_{k=1}^\infty
\Sigma_D(i,\alpha_l)\cap \{H_{i+k}=H_i\}\right).\end{split}\]
The set $B_\alpha\setminus\{\alpha_1,\ldots,\alpha_m\}^\mathbb Z$ (resp. $B_\beta\setminus\{\beta_1,\ldots,\beta_m\}^\mathbb Z$) consists of the set of sequences in $\Sigma_D$ such that any right (resp. left) bracket in the sequence is closed. 
Define $\phi_\alpha\colon B_\alpha\to \Sigma_\alpha$ by
\[(\phi_\alpha(\omega))_i=\begin{cases}
  \alpha_j&\text{ if }\omega_i=\alpha_j,\ j\in\{1,\ldots,m\},\\
  \beta&\text{ if }\omega_i\in\{\beta_1,\ldots,\beta_m\}.
\end{cases}\]
In other words, $\phi_\alpha(\omega)$ is obtained by replacing all $\beta_j$, $j\in\{1,\ldots,m\}$ in $\omega$ by $\beta$.
Similarly, define $\phi_\beta\colon B_\beta\to \Sigma_\beta$ by
\[(\phi_\beta(\omega))_i=\begin{cases}
    \alpha&\text{ if }\omega_i\in\{\alpha_1,\ldots,\alpha_m\},\\
    \beta_j&\text{ if }\omega_i=\beta_j,\ j\in\{1,\ldots,m\}.
\end{cases}\]
Clearly, $\phi_\alpha$, $\phi_\beta$ are continuous, injective and commute with the left shifts.
The next lemma summarizes what was proved in
\cite[Section~4]{Kri74}. 

\begin{lemma}\label{null}
For each $\gamma\in\{\alpha,\beta\}$ the following hold:
\begin{itemize}
\item[(a)]
$\lambda_\gamma(
\phi_\gamma({B}_\gamma))=1$.
\item[(b)] the map 
 $\psi_\gamma\colon \phi_\gamma(B_\gamma)\to B_\gamma$ satisfying
$\psi_{\gamma}\circ\phi_\gamma=1_{B_\gamma}$ is continuous.
\end{itemize}
\end{lemma}

\begin{proof}We reproduce the proof in \cite[Section~4]{Kri74}
for the reader's convenience.
 For $i\in\mathbb Z$ we define $\widehat H_{\alpha,i}\colon \Sigma_\alpha\to\mathbb Z$ and $\widehat H_{\beta,i}\colon \Sigma_\beta\to\mathbb Z$  by
      \[\begin{split}\widehat H_{\alpha,i}(\omega)&=\begin{cases}\sum_{j=0}^{i-1} \sum_{l=1}^{m}(\delta_{{\alpha_l},\omega_j}-\delta_{\beta,\omega_j})&\text{ for }  i\geq1,\\\sum_{j=i}^{-1} \sum_{l=1}^{m}(\delta_{{\beta},\omega_j}-\delta_{\alpha_l,\omega_j})&\text{ for } i\leq -1,\\
      0&\text{ for }i=0,\ \text{and}\end{cases}\\
            \widehat H_{\beta,i}(\omega)&=\begin{cases}\sum_{j=0}^{i-1} \sum_{l=1}^{m}(\delta_{{\alpha},\omega_j}-\delta_{\beta_l,\omega_j})&\text{ for }  i\geq1,\\\sum_{j=i}^{-1} \sum_{l=1}^{m}(\delta_{{\beta_l},\omega_j}-\delta_{\alpha,\omega_j})&\text{ for } i\leq -1,\\
      0&\text{ for }i=0,\end{cases}\end{split}\]
      respectively.
      Clearly we have
\[\begin{split}\phi_\alpha(B_\alpha)&=\bigcap_{i=-\infty}^\infty\bigcup_{l=1}^m\left(
\Sigma_\alpha(i,\alpha_l)\cup\bigcup_{k=1}^\infty
\Sigma_\alpha(i,\beta)\cap\{  \widehat H_{\alpha,i+1-k}=\widehat H_{\alpha,i+1} \}\right),\\
\phi_\beta(B_\beta)&=\bigcap_{i=-\infty}^\infty\bigcup_{l=1}^m\left(
\Sigma_\beta(i,\beta_l)\cup\bigcup_{k=1}^\infty
\Sigma_\beta(i,\alpha)\cap \{\widehat H_{\beta,i+k}=\widehat H_{\beta,i}\}\right).\end{split}\]

Let $\omega\in \Sigma_\alpha\setminus
\phi_\alpha(B_\alpha)$.
Then there exists $i\in\mathbb Z$ such that
$\omega_i=\beta$ and
$\widehat H_{\alpha,i+1-k}(\omega)\neq\widehat H_{\alpha,i+1}(\omega)$ for $k\geq1$.
Then $\omega_{i-1}=\beta$,
for otherwise 
$\widehat H_{\alpha,i-1}(\omega)=\widehat H_{\alpha,i+1}(\omega)$.
If $\omega_{i-2}\neq\beta$ then 
$\omega_{i-3}=\beta$, for otherwise
$\widehat H_{\alpha,i-3}(\omega)=\widehat H_{\alpha,i+1}(\omega)$. Iterating this argument gives
$\liminf_{n\to\infty}\frac{1}{n}\#\{i-n+1\leq j\leq i\colon\omega_j=\beta\}\geq\frac{1}{2}.$
Birkhoff's ergodic theorem for $(\sigma^{-1},\lambda_{\alpha})$ yields
$\lambda_\alpha(\phi_\alpha(B_\alpha))=1$, as required in Lemma~\ref{null}(a).
The same reasoning 
for $(\sigma,\lambda_{\beta})$ yields
$\lambda_\beta(\phi_\beta(B_\beta))=1$.

The map $\psi_\alpha\colon \phi_\alpha(B_\alpha)\to B_\alpha$ in Lemma~\ref{null}(b) is given by
\[(\psi_\alpha(\omega))_i=\begin{cases}
  \alpha_j&\text{ if }\omega_i=\alpha_j,\ j\in\{1,\ldots,m\},\\
  \beta_j&\text{ if }\omega_i=\beta,\ \omega_{r_\alpha(i,\omega)}=\beta_j,\ j\in\{1,\ldots,m\},
\end{cases}\]
where 
$r_\alpha(i,\omega)=\max\{j<i+1\colon \widehat H_{\alpha,j}(\omega)=\widehat H_{\alpha,i+1}(\omega)\}$.
The continuity is obvious.
Similarly, the map $\psi_\beta\colon \phi_\beta(B_\beta)\to B_\beta$ is given by
\[(\psi_\beta(\omega))_i=\begin{cases}
   \alpha_j&\text{ if }\omega_i=\alpha,\ \omega_{r_\beta(i,\omega)}=\alpha_j,\ j\in\{1,\ldots,m\},\\
    \beta_j&\text{ if }\omega_i=\beta_j,\ j\in\{1,\ldots,m\},
\end{cases}\]
where 
$r_\beta(i,\omega)=\min\{j>i\colon \widehat H_{\beta,j}(\omega)=\widehat H_{\beta,i}(\omega)\}$. The continuity is obvious.
\end{proof}
By Lemma~\ref{null},  
the measures 
 \begin{equation}\label{ergmme}\nu_{\alpha}=\lambda_\alpha\circ\phi_\alpha\text{ and } \nu_\beta=\lambda_\beta\circ\phi_\beta\end{equation}
are Bernoulli of entropy $\log(m+1)$,
and charge any non-empty open set.
 Using the relations in Section~\ref{Dyck} one can check that
$A_0\cup A_\alpha\subset B_\alpha$ and
$A_0\cup A_\beta\subset B_\beta$.
Hence, the ergodic mmes for $\Sigma_D$ are exactly $\nu_\alpha$ and $\nu_\beta$.

\subsection{The range of the coding map}\label{range}
Since it is important to know which measures in $\mathcal M(\sigma)$ can be transferred to measures in $\mathcal M(f)$, we need to clarify the range of
 $\bar\pi$.
\begin{lemma}
\label{cset}
We have 
\[\Sigma_D\setminus\bar\pi(X)=
\left(\bigcup_{N=0}^\infty\bigcap_{n=N}^\infty\sigma^{n}\left(\bigcup_{i=1}^m[\alpha_i]\right)\right)\cup\left(\bigcup_{N=0}^\infty\bigcap_{n= N}^\infty\sigma^{-n}\left(\bigcup_{i=1}^{m}[\beta_i]\right)\right).\]
\end{lemma}
\begin{proof}
Let $\omega=(\omega_n)_{n\in\mathbb Z}\in \Sigma_D\setminus\bar\pi(X)$.
For each $n\geq1$ we write 
$\bigcap_{k=0}^{n-1} f^{-k}({\rm int}(\Omega_{\omega_k}))=I_n\times J_n^+\times (0,1)$
and 
$\bigcap_{k=0}^{1-n} f^{-k}({\rm int}(\Omega_{\omega_k}))=(0,1)\times J_n^-\times K_n,$
where $I_n$, $J_n^+$, $J_n^-$, $K_n$ are open intervals in $(0,1)$.
Since $\omega\notin\bar\pi(X)$ we have 
$\bigcap_{n=1}^\infty(I_n\times (J_n^+\cap J_n^-)\times K_n)=\emptyset$.
So, either $\bigcap_{n=1}^\infty I_n$, $\bigcap_{n=1}^\infty (J_n^+\cap J_n^-)$ or $\bigcap_{n=1}^\infty K_n$ is an empty set.
For $n\geq1$
let $I_n^l$, $I_n^r$
denote the endpoints of $I_n$ with
$I_n^l<I_n^r$.
Similarly we define $J_n^{\pm,l}$, $J_n^{\pm,r}$ and $K_n^l$, $K_n^r$ which are endpoints of $J_n^{+}$, $J_n^{-}$, $K_n$.
We say a real sequence 
$\{a_n\}_{n=1}^\infty$ is {\it e.c.} (eventually constant) if $a_n=a_{n+1}$ holds for all sufficiently large $n$.
We use the following elementary observation that for a nested decreasing sequence $\{(a_n,b_n)\}_{n=1}^\infty$ of open intervals
with $b_n-a_n\to0$,
$\bigcap_{n=1}^\infty (a_n,b_n)=\emptyset$ holds
if and only if
$\{a_n\}_{n=1}^\infty$ or
$\{b_n\}_{n=1}^\infty$ is e.c.
There are three cases.
\medskip

\noindent{\it 
Case 1. $\bigcap_{n=1}^\infty I_n=\emptyset$.} 
If $\{I_n^l\}_n$ is e.c., then $\omega_n=\alpha_1$ for all sufficiently large $n$.
If
$\{I_n^r\}_n$ is e.c., then $\omega_n\in\{\beta_1,\ldots,\beta_m\}$ for all sufficiently large $n$.

\smallskip

\noindent{\it 
Case 2. $\bigcap_{n=1}^\infty (J_n^+\cap J_n^-)=\emptyset$}.
Then 
$\bigcap_{n=1}^\infty J_n^+=\emptyset$ or
$\bigcap_{n=1}^\infty J_n^-=\emptyset$.
If $\{J_n^{+,l}\}_n$ is e.c., then
$\omega_n=m+1$
for all sufficiently large $n$.
If $\{J_n^{+,r}\}_n$ is e.c., then $\omega_n=\beta_m$ for all sufficiently large $n$. 
If $\{J_n^{-,l}\}_n$ is e.c., then
$\omega_{-n}=\alpha_1$ for all sufficiently large $n$.
If $\{J_n^{-,r}\}_n$ is e.c.,
then
$\omega_{-n}=\alpha_m$
for all sufficiently large $n$.
\smallskip

\noindent{\it 
Case 3. $\bigcap_{n=1}^\infty K_n=\emptyset$.}
If $\{K_n^l\}_n$ is e.c., then
$\omega_{-n}\in\{\alpha_1,\ldots,\alpha_m\}$ for all sufficiently large $n$.
If $\{K_n^r\}_n$ is e.c.,
then
$\omega_{-n}=\beta_m$ for all sufficiently large $n$.
\medskip

\noindent In either of the cases, $\omega$ is contained in the right set in the equality in the lemma. 

To show the reverse inclusion, suppose that $\omega\in\Sigma_D$ is contained in the right set of the desired equality.
If there were $p\in X$
such that $\bar\pi(p)=\omega$,
there would be some $n\in\mathbb Z$ such that $f^n(p)\notin\bigcup_{i=1}^{2m}{\rm int}(\Omega_i)$, a contradiction.
Hence $\omega\notin\bar\pi(X)$.
 \end{proof}

From Lemma~\ref{cset} and Poincar\'e's recurrence theorem
we obtain the following.

\begin{cor}\label{cset-cor}
If $\nu\in\mathcal M(\sigma)$ satisfies $\nu(\bigcup_{i=1}^m
[\alpha_i])\nu(\bigcup_{i=1}^{m}[\beta_i])>0$, then 
 $\nu(\bar\pi(X))=1$.
\end{cor}

 \subsection{Lyapunov exponents}\label{lyapunov-sec}
  Define $\varphi^u\colon \Sigma_{D}\to\mathbb R$ and $\varphi^c\colon \Sigma_{D}\to\mathbb R$ by 
  \[\begin{split}\varphi^u(\omega)&=\begin{cases}-\log a&\text{ on }  \bigcup_{i=1}^{m}[\alpha_i],\\
 -\log(1-ma)&\text{ on }  \bigcup_{i=1}^{m}[\beta_i],\ \text{and}\end{cases}\\
 \varphi^c(\omega)&=-H_1(\omega)\log m,\end{split}\]
respectively. For a measure $\nu\in\mathcal M(\sigma)$
and $t=u,c$ we put
$\chi^t(\nu)=\int\varphi^td\nu$.
We say $\omega\in \Sigma_{D}$ is {\it regular} if for each $t=u,c$
 there exists $\chi^t(\omega)\in\mathbb R$ such that
 \[\lim_{n\to\infty}\frac{1}{n} \sum_{k=0}^{n-1}\varphi^t(\sigma^k\omega)=
\lim_{n\to\infty}\frac{1}{n} \sum_{k=0}^{n-1}\varphi^t(\sigma^{-k}\omega) =\chi^t(\omega).\]
Since $\varphi^u$, $\varphi^c$ are continuous,
they induce a cocycle of $2\times2$ matrices over $(\Sigma_{D},\sigma)$.
 By Oseledec's theorem  \cite{Ose68},
 for each $\nu\in \mathcal M(\sigma)$ 
 the set of regular points has full $\nu$-measure. 
  For $\mu\in \mathcal M(f)$ with $\mu(X)=1$ and $t=u,c$ we put $\chi^t(\mu)=\chi^t(\mu\circ\bar\pi^{-1})$, and
call them {\it unstable, central Lyapunov exponents of $\mu$} respectively.
 If $\mu$ is ergodic,
 $\chi^u(\mu)$,  $\chi^c(\mu)$
equal the
exponential growth rates of the diagonal elements of the Jacobian matrices of $f$
in the $x,y$-directions along $\mu$-typical orbits respectively.




\subsection{Entropy bounds}\label{entropy}

Ruelle's inequality \cite[Theorem~2,b)]{Rue78} asserts that the entropy of an invariant Borel probability measure for a $C^1$ diffeomorphism of a compact Riemannian manifold does not exceed the integral of the sum of non-negative pointwise Lyapunov exponents.
As $f$ is not a diffeomorphism, we slightly modify the proof of
\cite[Theorem~2,b)]{Rue78} and obtain a version of Ruelle's inequality. 
\begin{lemma}\label{ruelle-ineq}
Let $\mu\in \mathcal M(f)$ be ergodic and satisfy $\mu(X)=1$. 
 Then $h(\mu)\leq\chi^u(\mu)+\max\{\chi^c(\mu),0\}.$
\end{lemma}
\begin{proof}
We may assume $\mu$ is not supported on a periodic orbit, for otherwise the desired inequality is obvious.
We show that
for each integer $N\geq1$ there exists a finite partition $\mathcal Q_N$ of $[0,1]^3$ into Borel sets with the following properties: 

\begin{itemize}
\item[(a)] $\lim_{N\to\infty}\sup\{{\rm diam}(Q)\colon Q\in\mathcal Q_N\}=0$; 

\item[(b)] 
$\mu(\partial Q)=0$ for $Q\in\mathcal Q_N$;

\item[(c)] (cf.\cite[p.86~(a)]{Rue78}) there exists $C>0$, and
 for each $n\geq1$
there is $N(n)\geq1$ such that if $N\geq N(n)$, then for $Q\in\mathcal Q_N$ and $\omega\in\bar\pi(Q)$,
\[\#\{R\in \mathcal Q_N\colon R\cap f^n(Q)\neq\emptyset\}\leq C\exp\left(\sum_{k=0}^{n-1}\varphi^u(\sigma^k\omega )+\max\left\{\sum_{k=0}^{n-1}\varphi^c(\sigma^k\omega ),0\right\}\right).\]

\end{itemize}
Then the rest of the proof of \cite[Theorem~2,b)]{Rue78}
goes through with $f$ 
without any essential modification, and yields the desired inequality in Lemma~\ref{ruelle-ineq}.

Put $a_m=\min\{a,1-ma\}$.
For an interval $I\subset [0,1]$
we denote by $|I|$ its Euclidean length.
Let $\mathbb Q(r)$ denote the set of
rational numbers in $(0,1)$ of the form
$\sum_{k=1}^n\frac{d_k}{r^k}$, $d_k\in\{0,1\}$ for some $n\geq1$.
For each $N\geq1$, there exists a finite partition $\mathcal Q_N$ of $[0,1]^3$ into Cartesian products $I\times J\times K$ of 
subintervals $I$, $J$, $K$ of $[0,1]$ with the following properties:
there exists 
$k\geq N$ such that $F_a^k|_{I}$ is affine and $|F_a^k(I)|=1$,
  $a_m^{N+1}\leq |I|< a_m^N$;
the endpoints of $J$ are in $\mathbb Q(r)\cup\{0,1\}$, and
$\frac{1}{m^{N+1}}\leq |J|< \frac{1}{m^{N}}$;
$|K|=\frac{1}{m^{N}}$; $f^n|_{I\times J\times K}$ is affine for $1\leq n\leq N$.
Then (a) is obvious. For each $Q\in\mathcal Q_N$,
any ergodic measure in $\mathcal M(f)$ which gives measure $1$ to $\partial Q$ is supported on a periodic orbit. Hence (b) follows from the assumption on $\mu.$
To verify (c),
let $n\geq1$, $N\geq n$, $Q=I\times J\times K\in\mathcal Q_N$.
Then $f^n|_Q$ is affine,
and if $f^n(Q)=I'\times J'\times K'$ then
$|I'|=|I|\exp\sum_{k=0}^{n-1}\varphi^u(\sigma^k\omega)$,
$|J'|=|J|\exp\sum_{k=0}^{n-1}\varphi^c(\sigma^k\omega)$ for $\omega\in\bar\pi(Q)$ and $|K'|<|K|$. Hence, 
the number of elements of $\mathcal Q_N$ which intersect $f^n(Q)$
is bounded from above by
\[\left(a_m^{-1}\exp\sum_{k=0}^{n-1}\varphi^u(\sigma^k\omega)+2\right)\times
\left(m\exp\left\{\sum_{k=0}^{n-1}\varphi^c(\sigma^k\omega),0\right\}+2\right)\times 2,\]
for all $\omega\in\bar\pi(Q)$.
Taking $C=2a_m^{-1}m$ we obtain (c).
\end{proof}
  
  Lemma~\ref{ruelle-ineq} provides an upper bound on the entropy of an ergodic measure in terms of its Lyapunov exponents. 
 The next lemma provides an upper bound on
 the unstable Lyapunov exponent of an ergodic measure.
 Let $\delta\geq0$.
We say a measure $\mu\in\mathcal M(f)$ is {\it $\delta$-biased}
if $\mu(X)=1$ and $|\mu\left(\bigcup_{i=1}^m\Omega_i\right)-\mu\left(\bigcup_{i=m+1}^{2m}\Omega_i\right)|\leq\delta.$
Note that $\mu\in\mathcal M(f)$ 
with $\mu(X)=1$ is $\delta$-biased if and only if
$|\chi^c(\mu)|\leq\delta\log m$.
For $a\in(0,c_0)$ we put 
\[H(a)=-\log\sqrt{a(1-ma)}.\]
This number is the unstable Lyapunov exponent of 
$0$-biased ergodic measures.
Note that $H(c_1)=H(c_2)=\log(m+1)$
and $H(a)<\log (m+1)$ for $a\in(c_1,c_2)$.

\begin{lemma}\label{zero-u}
 Let $\delta>0$,
and let $\mu\in \mathcal M(f)$ be ergodic, $\delta$-biased and satisfy $\mu(X)=1$.
 Then $\chi^u(\mu)\leq(1+2\delta)H(a).$
\end{lemma}
\begin{proof}
From the ergodicity of $\mu\circ\bar\pi^{-1}$ and 
 Birkhoff's ergodic theorem, for any $\epsilon>0$ there exists $n(\epsilon)\geq1$ such that for any $n\geq n(\epsilon)$ there exists a finite subset  $\Gamma_n$ of $ L_n(\Sigma_{D})$ such that for  $\gamma_0\cdots \gamma_{n-1}\in\Gamma_n$ and  $\omega\in[\gamma_0\cdots \gamma_{n-1}]$ we have 
 \begin{equation}\label{lambda-ua}\left|\frac{1}{n}\sum_{k=0}^{n-1}\1_{\bigcup_{i=1}^{m} [\alpha_i]}(\sigma^k\omega )-\frac{1}{2}\right|<\delta\quad\text{and}\quad\left|\frac{1}{n}\sum_{k=0}^{n-1}\varphi^u(\sigma^k\omega)- \chi^u(\mu)\right|<\epsilon,\end{equation}
 where $\1_{\bigcup_{i=1}^{m} [\alpha_i]}$ denotes
 the indicator function of 
 $\bigcup_{i=1}^{m} [\alpha_i]$.
 The first inequality in \eqref{lambda-ua} implies
\[\begin{split}\frac{1}{n}\sum_{k=0}^{n-1}\varphi^u(\sigma^k\omega)&\leq\left(\frac{1}{2}+\delta\right)\left(\log\frac{1}{a}+\log\frac{1}{1-ma}\right)=\left(1+2\delta\right)H(a).\end{split}\] 
Combining this with the second inequality in
 \eqref{lambda-ua} yields 
 $\chi^u(\mu)<(1+2\delta)H(a)+\epsilon.$
 Since $\epsilon>0$ is arbitrary, the desired inequality in the lemma holds.
    \end{proof}

Combining the previous two lemmas we obtain a criterion for the non-entropy approachability.
    \begin{lemma}\label{ent-ap}
If $a\in(c_1,c_2)$
 and 
 $\delta>0$, then
 any non-ergodic $\delta$-biased measure $\mu\in \mathcal M(f)$ with $h(\mu)>(1+2\delta)H(a)+\delta\log m$ is not entropy approachable by ergodic measures. 
\end{lemma}
\begin{proof}
Let $\mu\in\mathcal M(f)$ and $\delta>0$ be as in the lemma.
Since $\mu$ is $\delta$-biased
we have $|\chi^c(\mu)|\leq\delta\log m.$
   Let $\{\mu_n\}_{n=1}^\infty$ be a sequence of ergodic measures in $\mathcal M(f)$ with positive entropy 
   such that
   $\mu_n\to\mu$. 
   Then we have $\mu_n(X)=1$ for $n\geq1$.
    Since $\mu$ is $f$-invariant, it does not charge 
  the wandering set 
    $\{(x,y,z)\in[0,1]^3\colon x=ma\}$ that is the discontinuity of $\varphi^c\circ\bar\pi$. Hence we have $\lim_{n\to\infty}\chi^c(\mu_n)=\chi^c(\mu)$.
    In particular, for any $\epsilon>0$ there exists $N\geq1$ such that $\mu_n$ is $(\delta+\epsilon)$-biased for all $n\geq N$.
By Lemmas~\ref{ruelle-ineq} and \ref{zero-u},
    \[\begin{split}\limsup_{n\to\infty} h(\mu_n)&\leq\limsup_{n\to\infty}\chi^u(\mu_n)+\limsup_{n\to\infty}\max\{\chi^c(\mu_n),0\}\\
    &\leq(1+2\delta )
    H(a)+\delta\log m< h(\mu),\end{split}\] which means that
    $\mu$ is not entropy approachable
   by ergodic measures. 
   \end{proof}

\subsection{Approximation by periodic measures}\label{per-sec}
To show one part of Theorem~C, we need the next lemma that is based on the symmetry of the Dyck shift.

 \begin{lemma}\label{dyck-ap}
 There is a sequence $\{\nu_n\}_{n=1}^\infty$ of ergodic measures in $\mathcal M(\sigma)$ 
 such that each $\nu_n$ is supported on a periodic orbit,
 satisfies $\chi^c(\nu_n)=0$,
 and 
$\nu_n\to\frac{1}{2}(\nu_{\alpha}+\nu_{\beta})$.
\end{lemma}
\begin{proof}
We define $\rho\colon D\circlearrowleft$ by
$\rho (\alpha_i)=\beta_i$ and $\rho(\beta_i)=\alpha_i$
for $1\le i\le m$, and
 define an injective map $\Phi\colon \Sigma_D\to D^\mathbb Z$
by $\Phi((\omega_n)_{n\in\mathbb Z})=(\rho(\omega_{-n}))_{n\in\mathbb Z}$.
For $n\geq1$ and $\gamma_1\cdots \gamma_n\in D^n$ we set
$\rho^{\ast}(\gamma_1\cdots \gamma_n)=\rho(\gamma_n)\cdots \rho(\gamma_1)$. Later we will show that
\begin{itemize}
\item[(a)]
$\rho^{\ast}(\gamma_1\cdots \gamma_n)\in L(\Sigma_D)$, and

\item[(b)]
$(\gamma_1\cdots \gamma_n\rho^{\ast}(\gamma_1\cdots \gamma_n))^j\in L(\Sigma_D)$
for $j\ge 1$.    
\end{itemize}
Since $\nu_\alpha$ is ergodic, 
there exists
 $\omega\in\Sigma_D$ such that 
$\frac{1}{n}\sum_{k=0}^{n-1}\delta_{\sigma^k\omega }
\rightarrow \nu_{\alpha}$
where $\delta_{\sigma^k\omega}$ denotes
the unit point mass at $\sigma^k\omega$.
By (a), 
$\Phi$ induces a homeomorphism on $\Sigma_D$ 
and satisfies $\Phi\circ\sigma=\sigma^{-1}\circ\Phi$.
Moreover we have
   $h(\nu)=h(\nu\circ\Phi^{-1})$
    and $\nu(\bigcup_{i=1}^m[\alpha_i])
    =\nu\circ\Phi^{-1}(\bigcup_{i=1}^m[\beta_i])$
    for $\nu\in\mathcal M(\sigma)$.    
It follows that $\nu_{\beta}=\nu_{\alpha}\circ\Phi^{-1}$.
Then we have $\frac{1}{n}\sum_{k=0}^{n-1}\delta_{\sigma^{-k}(\Phi(\omega ))}
\rightarrow \nu_{\beta}$.
By (b),
for $n\ge 1$ the sequence
\[\zeta^{(n)}=\cdots  \omega_0\cdots\omega_{n-1}\rho^{\ast}(\omega_0\cdots\omega_{n-1})
\omega_0\cdots\omega_{n-1}\rho^{\ast}(\omega_0\cdots\omega_{n-1})\cdots\] 
in
$D^\mathbb Z$ is contained in $\Sigma_D$, and is a periodic point of $\sigma$.
Let $\nu_n\in\mathcal M(\sigma)$ denote the measure
supported on the orbit of $\zeta^{(n)}$.
Then we obtain $\chi^c(\nu_n)=0$ and $\nu_n\rightarrow\frac{1}{2}(\nu_{\alpha}+\nu_{\beta})$  as required in the lemma.

 By the relations in Section~\ref{Dyck}, for each $\lambda=\gamma_1\cdots \gamma_n\in L(\Sigma_D)$ 
either (i) ${\rm red}(\lambda)=1$, or (ii) 
${\rm red}(\lambda)=\xi\eta$
for some $\xi\in L(\{\beta_1,\ldots,\beta_m\}^{\mathbb{Z}})$ and
$\eta\in L(\{\alpha_1,\ldots,\alpha_m\}^{\mathbb{Z}})$. In case (i) we have
 ${\rm red}(\rho^{\ast}(\lambda ))=1$
and ${\rm red}((\lambda\rho^{\ast}(\lambda ))^j)=1$ for $j\ge 1$,
which implies (a) (b).
In case (ii) we have ${\rm red}(\rho^{\ast}(\lambda))=
\rho^{\ast}(\eta)\rho^{\ast}(\xi)$ and
 $\rho^{\ast}(\eta)\in
L(\{\beta_1,\ldots,\beta_m\}^{\mathbb{Z}})$,
 $\rho^{\ast}(\xi)\in
L(\{\alpha_1,\ldots,\alpha_m\}^{\mathbb{Z}})$, and so
 (a) holds. Since
${\rm red}(\rho^{\ast}(\xi)\xi)=
{\rm red}(\eta\rho^{\ast}(\eta))=1$
we have ${\rm red}((\lambda\rho^{\ast}(\lambda ))^j)=\xi\rho^{\ast}(\xi)$
for $j\ge 1$,
which implies (b).
\end{proof}


\subsection{Proof of Theorem~B}\label{pfthmb}
Calculations based on \eqref{ergmme} give
$\nu_\alpha(\bigcup_{i=1}^m[\alpha_i])=\frac{m}{m+1}$,
$\chi^c(\nu_\alpha)=-\frac{m-1}{m+1}\log m$ and
$\nu_\beta(\bigcup_{i=1}^{m}[\beta_i])=\frac{m}{m+1}$,
$\chi^c(\nu_\beta)=\frac{m-1}{m+1}\log m$.
From these equalities and
Corollary~\ref{cset-cor} we obtain
$\nu_\alpha(A_\alpha\cap \bar\pi(X))=1$ and $\nu_\beta(A_\beta\cap\bar\pi(X))=1$.

 By Lemma~\ref{m-bij}, $\bar\pi$
 induces a homeomorphism between
 $A_\alpha\cap\bar\pi(X)$ and $\bar\pi^{-1}(A_\alpha)$ that commutes with $\sigma$ and $f$. 
   Clearly,
   if $U$ is a non-empty open subset of $[0,1]^3$ then $\bar\pi(U)$ contains an open set.
Hence, the measure
 $\mu_\alpha=\nu_{\alpha}\circ\bar\pi$ is Bernoulli of entropy $\log (m+1)$,
  charges any non-empty open set, and satisfies
  $\mu_\alpha(\bigcup_{i=1}^m\Omega_i)=\frac{m}{m+1}$.
  By the same reasoning, we conclude that the measure
  $\mu_\beta=\nu_{\beta}\circ\bar\pi$ is Bernoulli of entropy $\log (m+1)$ and charges any non-empty open set, and satisfies $\mu_\beta(\bigcup_{i=m+1}^{2m}\Omega_i)=\frac{m}{m+1}$.

 Let
$a\in[c_1,c_2]$ and let
 $\mu\in\mathcal M(f)$ be ergodic.
 If $\mu(X)=0$ then $\mu$ is supported on a periodic orbit and $h(\mu)=0$. 
 Suppose $\mu(X)=1$. If $\mu(\bar\pi^{-1}(A_0))=0$,
then
$\mu(\bar\pi^{-1}(A_\alpha))=1$ or
$\mu(\bar\pi^{-1}(A_\beta))=1$.
From Lemma~\ref{m-bij} and the variational principle \cite{Bow75} for the Dyck shift, we obtain
$h(\mu)\leq\log(m+1)$.
If $\mu(\bar\pi^{-1}(A_0))=1$
then $\chi^c(\mu)=0$, and
Lemma~\ref{zero-u}
 gives $\chi^u(\mu)\leq H(a)$.
This together with Lemma~\ref{ruelle-ineq} yields
$h(\mu)\leq H(a)\leq\log(m+1)$.
The last inequality is strict if $a\in (c_1,c_2)$.
\qed

\subsection{Proof of Theorem~C}\label{pfthmc}
Let 
$\{\nu_n\}_{n=1}^\infty$ be a sequence in $\mathcal M(\sigma)$
for which the conclusion of Lemma~\ref{dyck-ap} holds.
Since each $\nu_n$ is approximated arbitrarily well in the weak* topology by an ergodic measure which is supported on a periodic orbit and has positive central Lyapunov exponent, 
 there exists a sequence $\{\tilde\nu_n\}_{n=1}^\infty$ of ergodic measures supported on a periodic orbit and satisfying
 $\tilde\nu_n(A\cap\bar\pi(X))=1$, $\tilde\nu_n\to \frac{1}{2}(\nu_{\alpha}+\nu_{\beta})$. 
We have 
$\frac{1}{2}(\nu_{\alpha}+\nu_{\beta})(A\cap\bar\pi(X))=1$, and so
by Lemma~\ref{m-bij},
 $\tilde\nu_n\circ\bar\pi$ is an ergodic measure in $\mathcal M(f)$  
satisfying $\tilde\nu_n\circ\bar\pi\to\frac{1}{2}(\nu_{\alpha}+\nu_{\beta})\circ\bar\pi=\frac{1}{2}(\mu_{\alpha}+\mu_{\beta})$.
We have verified that $\frac{1}{2}(\mu_{\alpha}+\mu_{\beta})$
is approachable by ergodic measures.
By Theorem~B, this measure is $0$-biased, and so not entropy approachable by Lemma~\ref{ent-ap}.\qed


\section*{Appendix: the Motzkin system}\label{Motzkin}
The Motzkin system is a version of the Dyck system.
For integers $m\geq2$ and $n\geq1$,
the $(m,n)$-Motzkin system is a subshift
of $m$ brackets with the standard bracket rule and $n$ additional symbols,
the units \cite{Ino10,M04}. 
The $(m,n)$-Motzkin system has exactly two ergodic measures \cite{Ino10} of maximal entropy $\log (m+n+1)$.
All arguments and results in this paper can be generalized to include the Motzkin system. Indeed,
for any $(m,n)$ 
there are piecewise affine maps on $[0,1]^2$ and $[0,1]^3$
whose coding spaces are isomorphic to the one and two-sided $(m,n)$-Motzkin shifts respectively. For such a map on $[0,1]^2$ in the case $(m,n)=(2,1)$, see FIGURE~3.

\begin{figure}
\begin{center}
\includegraphics[height=3.4cm,width=9.5cm]
{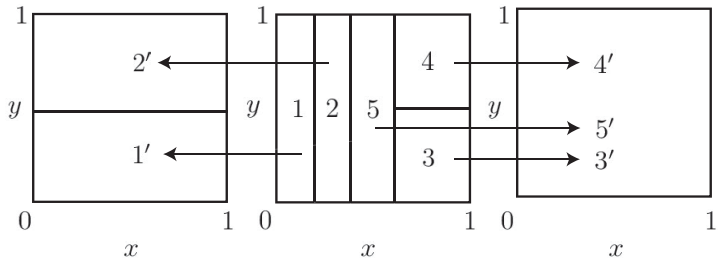}
\caption
{A piecewise affine map on $[0,1]^2$ whose coding space is isomorphic to the $(2,1)$-Motzkin system.}
\end{center}
\end{figure}

 \subsection*{Acknowledgments}
We thank Naoya Sumi for fruitful discussions. 
This research was supported by the JSPS KAKENHI 
19K21835, 20H01811, 21K03321.

\end{document}